\documentclass[a4paper,12pt]{article}
\usepackage{amsthm}
\usepackage{amssymb}
\usepackage{amsmath}
\usepackage{amsfonts}
\usepackage{url}
\usepackage{amsthm}
\newtheorem{theorem}{Theorem}
\newtheorem{lemma}{Lemma}
\newtheorem{corollary}{Corollary}

\newtheorem{definition}{Definition}[section]
\def\b1{\mbox{\boldmath $1$}}

\newcommand{\noi}{\noindent}

\newcommand{\be}{\begin{eqnarray}}
\newcommand{\ee}{\end{eqnarray}}

\numberwithin{mc}{subsection}

\numberwithin{md}{subsection} 
\numberwithin{mt}{subsection} 
\numberwithin{ml}{subsection} 
\numberwithin{mpro}{subsection} 

\title{{Kolmogorov's 0-1 law as false-truth law.}}
\author{A. Deshpande
\footnote {U.Strathclyde, Glasgow G40GE, Scotland, UK; e-mail: a.d.deshpande10@gmail.com}
}

\date{}

\begin{document}

\maketitle \baselineskip20pt
\parskip10pt
\parindent.4in

\begin{abstract} \noi The Kolmogorov's 0-1 law is so powerful and fundamental  in Probability theory that the result of the law of large numbers is its corollary. The latter is the theoretical bedrock of modern day statistical machine learning. Expected resurgence in the interest of the law's proof is challenged by its measure theoretic  nature. By observing 0-1 law as False-Truth law like in logic, we show here that the proof of the law is simplified. This simplicity comes from the fact that our  approach  hardly needs any measure theory beyond defining an obvious probability space.  A novel contribution via this article is the use of logic within the proof of a fundamental result in the theory of probability. 

\end{abstract}

\footnote{``This version of the article has been accepted for publication, after peer review (when applicable) but is not the Version of Record and does not reflect post-acceptance improvements, or any corrections. The Version of Record is available online at: http://dx.doi.org/10.1007/s13171-026-00445-w  "}
\section{Problem statement}\label{s1}
We assume that there is an underlying probability space ($\Omega, {\mathcal F}, \mathbf{P}$) with respect to which all the probability objects are defined. A collection $\{A_{\alpha}\}_{\alpha \in I}$ of events is said to be independent if for each $j,n \in I\subset \mathbb{N}$ and each distinct choice of $\alpha_1,\alpha_2,...,\alpha_j \in I$ , we have
\begin{equation}
\mathbf{P}(A_{\alpha_1}\cap A_{\alpha_2}\cap...\cap A_{\alpha_j})=\mathbf{P}(A_{\alpha_1})\mathbf{P}(A_{\alpha_2})...\mathbf{P}(A_{\alpha_j}).
\label{eq:s}
\end{equation}
\begin{definition}\label{1}
Given a sequence of events $A_1,A_2,...,$ we define their tail field by
\begin{equation}
\tau= \cap_{k=n}^{\infty}\sigma(A_{k+1},A_{k+2},...).
\label{1}
\end{equation}
\end {definition}
where $\sigma()$ as usual denotes the $\sigma$-algebra.
We now are in a position to state Kolmogorov's 0-1 law.
\begin{theorem}
\label{th:1}
If events $A_1,A_2,...$ are independent, with tail-field $\tau$, and if $A \in \tau$, then $\mathbf{P}(A)=0$ or 1.
\end{theorem}
Proof of Kolmogorov's 0-1 law is purely measure theoretic \cite{Kolmogorov}. It involves a ``bootstrapping procedure" in extending the equivalence of the probability measures from $\cup_{k=1}^{\infty}{\sigma(A_1,...,A_k})$ to $\sigma(A_1,A_2,...)$ . In the next section we describe two other approaches viz. algebraic and secondly via use of martingale theory.
\section{Two approaches}\label{s2}
One finds in Skorokhod \cite{Skorokhod}, an algebraic proof of the law which is pedagogically easier. As seen in \cite{Loeve}, in Skorokhod's approach, one is to observe that for the deduction of the law (viz. $ \mathbf{P}(A)=\{0,1\})$, the event $A$ has to be independent of itself. Since $\sigma(A_1,...,A_n)$ is independent of $\sigma(A_{n+1},A_{n+2},...)$, $\tau$ should be independent of the sigma algebra generated by $A_1,A_2,...$ and since $\tau$ being contained in $\sigma(A_1,A_2,...)$; it is to be independent of itself leading to the conclusion. In \cite{Port}, another approach based on the martingale theory is described. The argument goes as follows. First note that $\mathbf{P}(A|A_1,...,A_n)=\mathbf{P}(A)$. By a certain form of the martingale convergence theorem, we have, as $\lim_{n\rightarrow \infty}, \mathbf{P}(A|A_1,...,A_n)=\mathbf{1}_{A}$ a.e. Hence $\mathbf{P}(A)=\{0,1\}$.\\
\indent Expectedly, the above approaches also command heavy use of machinery from measure theory. This motivates us to prove the law using measure theory limited to only defining an obvious probability space. We pursue this challenge as follows. When the 0-1 law is viewed as False-Truth law,one can study independence between two events as a {\it independence relationship} between them. A collection of events validated by an {\it independence relation} forms what we call a {\it bracket class}, to which we apply a relaxed version of a result on set relations in \cite{Book}, specifically Theorem 2(b) Chapter 4 , thereby leading to the proof of the 0-1 law.  We elaborate on this program in the coming sections. However before we do so, we request the reader to remember the following standing assumption valid for the whole of the article.\\
{\bf Standing Assumption}. The statement of the Kolmogorov's 0-1 law cf. Theorem \ref{th:1} clarifies that the  law is valid only on independent events. Hence all events described in the article henceforth are assumed independent of each other. Expectedly the results employing these are valid only for independent events obviously enough to prove the 0-1 law. 
\section{Analysis}\label{s4}
We begin by defining the term ``Independence relation" $R$.
\begin{definition}\label{2.02}
Any subset $R$ of the cartesian product  consisting of elements of ${\mathcal F} \times {\mathcal F}$ is called an independence relation $R$, when defined as $R=\{(A,B): A,B \in {\mathcal F}, \mathbf{P}(A \cap B)=\mathbf{P}(A)\mathbf{P}(B)\}$.
\end {definition}
We henceforth describe $ARB$ as an independent relation between sets $A,B$. We  now formalize notion of $R$ being {\it symmetric} and {\it transitive}.\\
\begin{lemma}\label{2.2}
Let the standing assumption hold true. Then the independence relation $R$  is symmetric.
\end{lemma}
\begin{proof} Since the standing assumption holds we have events $A,B \in {\mathcal F}$ to be independent. By definition of  independence , we have, $\mathbf{P}(A \cap B)=\mathbf{P}(A)\mathbf{P}(B)=\mathbf{P}(B)\mathbf{P}(A)=\mathbf{P}(B\cap A)$. Thus $ARB \implies BRA$, and the conclusion follows.
\end{proof}
We now want to show that $R$ defined here is transitive i.e. on independent sets $A,B,C \in {\cal F}$; $ARB$ and $BRC$ implies $ARC$ i.e. $R$ is transitive. 
\begin{lemma}\label{2.3}
Let the standing assumption hold true. Then for events $A,B,C \in {\cal F}$ ,$R$ will contain $\{(A,B),(B,C),(A,C)\}$. Thus  $ARB,BRC$ implies $ARC$, i.e. independence relation $R$  is transitive.
\end{lemma}
\begin{proof} Since the standing assumption holds true we have $A,B,C$ as  independent events. Thus by equation \ref{eq:s}  we have
\begin{eqnarray*}
\mathbf{P}(A \cap B \cap C)&=&\mathbf{P}(A)\mathbf{P}(B)\mathbf{P}(C)\\
\mathbf{P}(A \cap B)&=&\mathbf{P}(A)\mathbf{P}(B)\\
\mathbf{P}(B\cap C)&=&\mathbf{P}(B)\mathbf{P}(C)\\
\mathbf{P}(A \cap C)&=&\mathbf{P}(A)\mathbf{P}(C)
\end{eqnarray*}
From Definition \ref{2.02},  $R$  hence contains $\{(A,B),(B,C),(A,C)\}$. This is transitivity. 
\end{proof}
{\bf Remark}\\
 We note here that $R$ is not transitive if the assumption of independence between all events $A,B$ and $C$ is dropped. A popular example to cite is that if $A=C$ and $ARB$ and $BRC$ holds, then $ARC$ does not hold in general. This example illustrates the need  of the standing assumption.

We formally define the concept of a Bracket class.
\begin{definition}\label{equ}
Let $\hat{R}$ be the  relation amongst events which are subsets of ${\mathcal F} $. For $A \in {\mathcal F}$, we define the Bracket class $[A]$ by

\begin{equation}\label{eq}
[A]=\{x:x \in {\mathcal F}, x\hat{R}A\}
\end{equation}
\end{definition}
We discuss some useful results related to the bracket class before we head to proving the law.

\begin{lemma}
Let ${\hat{R}}$ be an symmetric and transitive relation  consisting of subsets of ${\cal F} \times {\cal F}$ and let $a$ and $b$ be any two subsets of ${\cal F}$. Let the bracket classes $[a]$ and $[b]$ exist. If $a \in [b]$ then $[a]=[b]$.
\end{lemma}
\begin{proof} 
Since $a \in [b]$ we have both viz. $a\hat{R}b$ by definition of $[b]$ and secondly $b\hat{R}a$ (by virtue of $\hat{R}$ being symmetric) \\
(i) We show that $[a]\subseteq[b]$\\
Let $x \in [a]$ therefore $x\hat{R}a$ (by definition of $[a]$) but $a\hat{R}b$ (since $a \in [b]$). Therefore $x\hat{R}b$($\hat{R}$ by assumption in the statement is transitive). Therefore $x \in [b]$. Hence $[a]\subseteq[b]$. (ii) Similarly one can show from $b\hat{R}a$ that 
$[b] \subseteq[a]$. (i) and (ii) implies $[a]=[b]$.
\end{proof}
We require the following standard result from a first course in discrete mathematics. Since the statement is tuned for the current problem, we detail its proof.
\begin{theorem}\label{th:mt}
Let $\hat{R}$ be a symmetric and transitive relation consisting of subsets of ${\cal F} \times {\cal F}$ and the bracket class exists on a subset of ${\cal F}$. Let $a$ and $b$ be these  two subsets of ${\cal F}$. Then $[a]=[b]$ or $[a] \cap [b]= \emptyset$.
\end{theorem}
\begin{proof}
We need only show that if $[a] \cap [b] \neq \emptyset$ then $[a]=[b]$. To do this $y \in [a] \cap [b] $. Then $y \in [a] \implies [y]=[a]$ and $y \in [b] \implies [y]=[b]$ using above lemma. This concludes that $[a]=[b]$.
\end{proof}
\indent We now prove the Kolmogrov's law in Theorem \ref{th:1} using the concept of ``independence relation".
\begin{proof}
 Note  that $A \in \tau \subset \sigma(A_{n+1},...) \subset \sigma(A_1,A_2,...) $, we have $A \in \sigma(A_{n+1},A_{n+2},...) \forall$  non-negative integers $n$. Since $\sigma(A_1,...A_n)$ is independent of $\sigma(A_{n+1},...)$ this implies that $B_n \in \sigma(A_1,...A_n)$ is independent of $A$. Or in other words $B_nRA$ where $R$ is our usual independence relation.  Via the assumption of probability triplet, the existence of empty set is guaranteed. Hence from Definition \ref{equ}, the bracket class exists.  From Lemma \ref{2.2} $R$ has symmetry , Lemma \ref{2.3}, $R$ has transitivity. Hence from Theorem \ref{th:mt} we have $[A]=[B_n]$ for any  $n \geq 1$.  Therefore $[A]=\{\emptyset, \Omega,A,A^c\}$ must be true, where $A^c$ is complement of $A$. As  $A \in [A]$ we get $P(A \cap A)=P(A)=P(A)*P(A)$ implying $P(A)=\{0,1\}$.
\end{proof}

\end{document}